\documentclass[12pt,twoside]{amsart}
\usepackage{amsfonts}
\usepackage{amsmath}
\usepackage{amscd}
\usepackage{amssymb}
\usepackage{hyperref}
\usepackage{geometry}
\usepackage{graphicx}
\usepackage{inputenc}
\usepackage{enumerate}

\newtheorem{theorem}{Theorem}

\newtheorem{definition}[theorem]{Definition}
\newtheorem{example}[theorem]{Example}

\newtheorem{lemma}[theorem]{Lemma}

\newtheorem{proposition}[theorem]{Proposition}
\newtheorem{remark}[theorem]{Remark}

\begin{document}

\title{Almost Automorphic Hyperfunctions }
\author{Chikh BOUZAR and Amel BOUDELLAL }
\address{ Laboratory of Mathematical Analysis and Applications. Oran University 1 A.B.B., Algeria.}
\email{ch.bouzar@gmail.com}
\date{}

\begin{abstract}
Almost automorphy in the context of hyperfunctions is the main aim of this
work. We give different equivalent definitions of almost automorphic
hyperfunctions and then we study this class of hyperfunctions.
\end{abstract}
\maketitle
\section{Introduction}

Almost automorphy is due to S. Bochner, it appeared for the first time in
\cite{Boch56}, a study of almost automorphic classical functions is tackled
in his papers \cite{Boch61},\cite{Boch62},\cite{Boch64},
 where some of their main
properties are given. A more general study of almost automorphy is done in
\cite{v}. An example of an almost automorphic function which is not almost
periodic is given in \cite{v2}, which is a confirmation that almost
automorphy is more general than almost periodicity studied by H. Bohr, see
\cite{bohr}.

Almost automorphy in the framework of Sobolev-Schwartz distributions has
been tackled in the paper \cite{BZR}, where the problem of existence of
distributional almost automorphic solutions of linear differential equations
is considered. The class of almost automorphic distributions extends almost
periodic distributions introduced and studied by L. Schwartz in \cite{s} to
generalize the Stepanoff almost periodic functions \cite{stepa}. In the same
way as the papers \cite{BKT},\cite{BZR} the work \cite{kost} deals with
almost automorphic\ non quasi-analytic ultradistributions extending the
almost periodicity for ultradistributions, the subject of the work \cite{cio}%
, where an application to the Dirichlet problem is done. It is worth noting
that concept of ultradistributions is strictly larger than the one of
Sobolev-Schwartz distributions, see \cite{koma}.

Hyperfunctions of M. Sato \cite{sato},\cite{sato2}, extend the concepts of
classical functions, distributions and ultradistributions. They are
important tools in the study of ordinary and partial differential equations.
Fourier hyperfunctions are the subject of the paper \cite{ka}. Almost
periodic Fourier hyperfunctions, using the approach to hyperfunctions
proposed by T. Matsuzawa in \cite{matzu1},\cite{matzu2},\cite{matzu3}, are the subject
of the work \cite{Fap}.

The main aim of this work is to introduce and to study almost automorphy in
the context of Fourier hyperfunctions. In a forthcoming work, the
existence of almost automorphic hyperfunction solutions of some
functional-differential equations will be considered.

The paper is organised as follows : in the second section we first recall
the definition and some properties of almost automorphic classical functions, and also
the function space $\mathcal{D}_{L^{p}}^{\left\{ 1\right\} }$, then we
introduce the space of real analytic almost automorphic functions and give
their main properties. In the third section, we recall the
space of test functions for Fourier hyperfunctions, the space of $L^{p}$%
-hyperfunctions and we show some of their important results, see \cite{Fap},\cite{
Perio},\cite{Gelfand},\cite{Pili2}. In the last section, we introduce almost automorphic
hyperfunctions by proving a principal theorem characterising the concept of
almost automoprphy in the context of hyperfunctions, and then a study of
their properties is given.

\section{Real analytic almost automorphic functions}

We consider complex valued functions defined on the whole space of real
numbers $%
\mathbb{R}
$. Let $\mathcal{C}_{b}$ denotes the space of continuous bounded functions
on $%
\mathbb{R}
$ endowed with the norm $\Vert .\Vert _{\infty }$ of uniform convergence on $%
\mathbb{R}$. It is well-known that $\ (\mathcal{C}_{b},\Vert .\Vert _{\infty
})$ is a Banach algebra.

\begin{definition}
(Bochner \cite{Boch61}) A complex-valued function $f$ defined and continuous
on $\mathbb{R}$ is called almost automorphic, if for any real sequence $%
(s_{m})_{m\in \mathbb{N}}\subset \mathbb{R}$, one can extract a subsequence $%
(s_{m_{k}})_{k}$ such that
\begin{equation}
g(x):=\lim\limits_{k\rightarrow +\infty }f{(x+s_{m_{k}})}\ \ \text{is
well-defined for every}\ \ x\in \mathbb{R},\text{ \ }  \label{caa11}
\end{equation}%
and \
\begin{equation}
\lim\limits_{k\rightarrow +\infty }{g(x-s_{m_{k}})}=f(x)\ \ \text{for every }%
\ \ x\in \mathbb{R}\text{ }  \label{caa2}
\end{equation}%
The space of almost automorphic functions on $\mathbb{R}$ is denoted by $%
\mathcal{C}_{aa}$.
\end{definition}

\begin{remark}
The function $g$ in the above definition is not necessarily continuous but
it is measurable and bounded, so locally integrable.
\end{remark}

\begin{remark}
\label{limlim}In \cite{Boch62} is proved that a continuous function$~f~$is
almost automorphic if and only if for any sequence of real numbers $\left(
s_{m}\right) _{m\in
\mathbb{N}
},$ one can extract a subsequence $\left( s_{m_{k}}\right) _{k}$ such that%
\begin{equation*}
\lim_{l\rightarrow +\infty }\lim_{k\rightarrow +\infty }f\left(
x+s_{m_{k}}-s_{m_{l}}\right) =f\left( x\right) ,\forall x\in \mathbb{R}
\end{equation*}
\end{remark}

For the following results see \cite{Boch61},\cite{Boch64},\cite{BZR},\cite{v}.

\begin{proposition}
\label{pro}

\begin{enumerate}
\item \label{pro1}The space $\mathcal{C}_{aa}$ is a Banach subalgebra of \ $%
\mathcal{C}_{b}$ invariant by translations.

\item \label{pro4}If $f\in \mathcal{C}_{aa}$ and $g\in L^{1},$ then the
convolution $f\ast g\in \mathcal{C}_{aa}.$

\item \label{pro5}Let $f\in \mathcal{C}_{aa}$, if $f^{\prime }$ exists and
is uniformly continuous, then $f^{\prime }\in \mathcal{C}_{aa}.$

\item \label{pro6}A primitive of an almost automorphic function is almost
automorphic if and only if it is bounded.
\end{enumerate}
\end{proposition}

For more details on the following spaces and their properties see \cite%
{Perio, Pili2}. Let $p\in \left[ 1,+\infty \right] ,$ and define
\begin{equation*}
\mathcal{D}_{L^{p},h}^{\left\{ 1\right\} }=\left\{ \varphi \ \in C^{\infty
}:\left\Vert \varphi \right\Vert _{L^{p},h}:=\sup_{_{j\in
\mathbb{Z}
_{+}}}\frac{\left\Vert \varphi ^{(j)}\right\Vert _{p}}{h^{j}j!}<\infty
\right\} ,h>0,
\end{equation*}%
then $\left( \mathcal{D}_{L^{p},h}^{\left\{ 1\right\} },\left\Vert
.\right\Vert _{L^{p},h}\right) $ is a Banach space. The space $\mathcal{D}%
_{L^{p}}^{\left\{ 1\right\} }$:=$\cup _{h>0}$ $\mathcal{D}%
_{L^{p},h}^{\left\{ 1\right\} }$ endowed with the topology of inductive
limit of the Banach spaces $\left( \mathcal{D}_{L^{p},h}^{\left\{ 1\right\}
}\right) _{h>0}$ is a $\left( DF\right) $-space.

\begin{definition}
The space of real analytic almost automorphic functions on $
\mathbb{R}$ is denoted and defined by
\begin{equation*}
\mathcal{B}_{aa}^{\left\{ 1\right\} }= \{f\in C^{\infty }\left(\mathbb{R}
\right) :\forall j\in \mathbb{Z}
_{+},f^{(j)}\in \mathcal{C}_{aa}\text{ }\text{and}\text{ } \exists h>0,\exists
C>0,\forall j\in
\mathbb{Z}
_{+},\|f^{(j)}\|_{\infty }\leq Ch^{j}j!\}
\end{equation*}
\end{definition}

\begin{proposition}

\label{baa1}Let $f\in C^{\infty },$ the following statements are equivalent
\end{proposition}

\begin{enumerate}
\item $f\in \mathcal{B}_{aa}^{\left\{ 1\right\} }.$

\item \label{baa12}$f\in \mathcal{C}_{aa}\cap \mathcal{D}_{L^{\infty
}}^{\left\{ 1\right\} }.$

\item For every sequence $\left( \rho _{m}\right) _{m\in \mathbb{N}}\subset
\mathbb{R}$, there exist a subsequence $\left( \rho _{m_{k}}\right) _{k}~$%
and $\tilde{f}\in \mathcal{D}_{L^{\infty }}^{\left\{ 1\right\} }$ such that$%
~ $for every $x\in \mathbb{R}$, for every $j\in \mathbb{Z}_{+},$ we have%
\begin{equation}
\tilde{f}~^{(j)}\left( x\right) =\lim\limits_{k\rightarrow +\infty
}f^{\left( j\right) }\left( x+\rho _{m_{k}}\right) \text{ and~}%
\lim_{k\rightarrow +\infty }\tilde{f}~^{(j)}\left( x-\rho _{m_{k}}\right)
=f^{\left( j\right) }\left( x\right)  \label{equ2}
\end{equation}
\end{enumerate}

\begin{proof}
$1.\Rightarrow 2.$is obvious.

$2.\Rightarrow 3.~$Let $f\in \mathcal{C}_{aa}\cap \mathcal{D}_{L^{\infty
}}^{\left\{ 1\right\} }$, then $\forall j\in \mathbb{Z}_{+},$ $\forall
\left( \rho _{m}\right) _{m\in \mathbb{N}}\subset \mathbb{R},\exists (\rho
_{m_{j,k}})_{k}\subset \left( \rho _{m}\right) _{m},\exists (\tilde{f}%
~_{j})_{j}\subset L^{\infty }$ such that
\begin{equation*}
\forall x\in \mathbb{R},\lim_{k\rightarrow +\infty }f^{\left( j\right)
}\left( x+\rho _{m_{j,k}}\right) =:\tilde{f}~_{j}\left( x\right) ~,\
\lim_{k\rightarrow +\infty }\tilde{f}~_{j}\left( x-\rho _{m_{j,k}}\right)
=f^{\left( j\right) }\left( x\right)
\end{equation*}%
Then $\forall $ $x\in \mathbb{R},\forall \epsilon >0,\exists \eta >0,\forall
k\in \mathbb{%
\mathbb{N}
},k\geq \eta ,\left\vert f^{\left( j\right) }\left( x+\rho _{m_{j,k}}\right)
-\tilde{f}~_{j}\left( x\right) \right\vert \leq \epsilon $, so we have
\begin{eqnarray*}
\left\vert \tilde{f}~_{j}(x)\right\vert &=&\left\vert \tilde{f}%
~_{j}(x)-f^{\left( j\right) }\left( x+\rho _{m_{j,k}}\right) +f^{\left(
j\right) }\left( x+\rho _{m_{j,k}}\right) \right\vert \\
&\leq &\epsilon +\left\Vert f^{\left( j\right) }\right\Vert _{\infty },
\end{eqnarray*}%
which gives that
\begin{equation}
\left\Vert \tilde{f}~_{j}\right\Vert _{\infty }\leq Ch^{j}j!  \label{*}
\end{equation}%
For every $n\in \mathbb{Z}_{+},$ there exists a subsequence $\left( \rho
_{m_{n,k}}\right) _{k}$ of the sequence $\left( \rho _{m}\right) _{m}$ such
that%
\begin{equation*}
\forall j\leq n,~\lim_{k\rightarrow +\infty }f^{\left( j\right) }\left(
x+\rho _{m_{n,k}}\right) =\tilde{f}~_{j}\left( x\right) ,\forall x\in
\mathbb{R}\text{,}
\end{equation*}%
this result is proved in \cite{BZR2}. The map $k\longmapsto m_{k,k}$ is
strictly increasing from $\mathbb{N}$ to $\mathbb{N}.$ The sequence $\left(
\rho _{m_{k,k}}\right) _{k},$ which we denote by $\left( \rho
_{m_{k}}\right) _{k},~$is extracted from the subsequences $\left( \rho
_{m_{j,k}}\right) _{k},~j\in \mathbb{Z}_{+},$ which is in fact extracted
from the sequence $\left( \rho _{m}\right) _{m}.$ Consequently,
\begin{equation*}
\lim_{k\rightarrow +\infty }f^{\left( j\right) }\left( x+\rho
_{m_{k}}\right) =\tilde{f}~_{j}\left( x\right) \text{ exists }\forall x\in
\mathbb{R},~\forall j\in \mathbb{Z}_{+}
\end{equation*}%
With the same steps we have that
\begin{equation*}
\lim_{k\rightarrow +\infty }\tilde{f}~_{j}\left( x-\rho _{m_{k}}\right)
=f^{\left( j\right) }\left( x\right) ,~\forall x\in \mathbb{R},~\forall j\in
\mathbb{Z}_{+}
\end{equation*}%
Let $\left( \sigma _{n}\right) _{n\in \mathbb{N}}$ be a sequence~of positive
reals numbers converging to zero and consider the sequence of functions $%
\left( \phi _{n,k}\right) _{n,k\in \mathbb{N}}$ defined on $\mathbb{R~}$by
the equality
\begin{equation}
\phi _{n,k}(.)=\frac{f\left( .+\rho _{m_{k}}+\sigma _{n}\right) -f\left(
.+\rho _{m_{k}}\right) }{\sigma _{n}}=\int_{0}^{1}f^{\prime }\left( .+\rho
_{m_{k}}+\theta \sigma _{n}\right) d\theta
\end{equation}%
Since $f\in \mathcal{D}_{L^{\infty }}^{\left\{ 1\right\} },$ then $f^{\prime
}$ s bounded and uniformly continuous~on $\mathbb{R},$ so%
\begin{equation*}
\lim\limits_{k\rightarrow +\infty }\lim\limits_{n\rightarrow +\infty
}\int_{0}^{1}f^{\prime }\left( .+\rho _{m_{k}}+\theta \sigma _{n}\right)
d\theta =\lim\limits_{n\rightarrow +\infty }\lim\limits_{k\rightarrow
+\infty }\int_{0}^{1}f^{\prime }\left( .+\rho _{m_{k}}+\theta \sigma
_{n}\right) d\theta ,
\end{equation*}%
which gives that $\forall x\in \mathbb{R},$%
\begin{equation*}
\tilde{f}~_{1}\left( x\right) =\lim\limits_{k\rightarrow +\infty
}\lim\limits_{n\rightarrow +\infty }\phi _{n,k}\left( x\right)
=\lim_{n\rightarrow +\infty }\lim_{k\rightarrow +\infty }\phi _{n,k}\left(
x\right) :=\tilde{f}~_{0}^{\prime }\left( x\right)
\end{equation*}%
\newline
By iterating to all derivatives and by $\left( \ref{*}\right) $ it follows
that $\left\Vert \tilde{f}~_{0}^{(j)}\right\Vert _{\infty }=\left\Vert
\tilde{f}~_{j}\right\Vert _{\infty }$ $\leq Ch^{j}j!,$ i.e. $\tilde{f}%
~_{0}\in $ $\mathcal{D}_{L^{\infty }}^{\left\{ 1\right\} }$ such that $%
\left( \ref{equ2}\right) $ hold.

$3.\Rightarrow 1.$ Let $3$ holds, then we have $\forall j\in
\mathbb{Z}
_{+},f^{(j)}\in \mathcal{C}_{aa}$ and $\forall $ $x\in \mathbb{R},\forall
\epsilon >0,\exists \eta >0,\forall k\in \mathbb{%
\mathbb{N}
},k\geq \eta ,\left\vert \tilde{f}~^{(j)}\left( x-\rho _{m_{k}}\right)
-f^{\left( j\right) }\left( x\right) \right\vert \leq \epsilon $, so we have
\begin{eqnarray*}
\left\vert f^{(j)}(x)\right\vert &=&\left\vert f^{\left( j\right) }(x)-%
\tilde{f}~^{(j)}\left( x-\rho _{m_{k}}\right) +\tilde{f}~^{(j)}\left( x-\rho
_{m_{k}}\right) \right\vert \\
&\leq &\epsilon +\left\Vert \tilde{f}^{\left( j\right) }\right\Vert _{\infty
},
\end{eqnarray*}%
which gives that
\begin{equation}
\left\Vert f^{(j)}\right\Vert _{\infty }\leq Ch^{j}j!,
\end{equation}%
we deduce that $f\in \mathcal{B}_{aa}^{\left\{ 1\right\} }.$
\end{proof}

We give some important properties of the space $\mathcal{B}_{aa}^{\left\{
1\right\} }.$

\begin{proposition}
\label{Baa1}

\begin{enumerate}
\item $\mathcal{B}_{aa}^{\left\{ 1\right\} }$ is a closed subalgebra of $%
\mathcal{D}_{L^{\infty }}^{\left\{ 1\right\} }$ stable by translations and
derivations.

\item \label{Baa12}$\mathcal{B}_{aa}^{\left\{ 1\right\} }\ast L^{1}\subset $
$\mathcal{B}_{aa}^{\left\{ 1\right\} }.$

\item Let $f\in \mathcal{B}_{aa}^{\left\{ 1\right\} }$ and $F$ its
primitive, then $F\in \mathcal{B}_{aa}^{\left\{ 1\right\} }$ if and only if $%
F$ is bounded.
\end{enumerate}
\end{proposition}

\begin{proof}
$\mathit{1.}$ It is easy to see that $\mathcal{B}_{aa}^{\left\{ 1\right\} }$
is stable by translations and derivation. Let $f_{1}$ and $f_{2}$ belong to $%
\mathcal{B}_{aa}^{\left\{ 1\right\} }$, then we have $f_{1}f_{2}\in
C^{\infty }$, and $\left( f_{1}f_{2}\right) ^{(l)}\in \mathcal{C}_{aa}$, $%
\forall l\in
\mathbb{Z}
_{+}$ by Proposition $\ref{pro}$-$\left( \ref{pro1}\right) .$ Furthermore, $%
\exists h_{i}>0$, $\exists C_{i}>0$, $\forall l\in
\mathbb{Z}
_{+}$,%
\begin{eqnarray*}
\left\Vert (f_{1}f_{2})^{(l)}\right\Vert _{\infty } &\leq &\sum_{k+j=l}\frac{%
l!}{k!j!}\left\Vert f_{1}^{(k)}\right\Vert _{\infty }\left\Vert
f_{2}^{(j)}\right\Vert _{\infty } \\
&\leq &C_{1}C_{2}\sum_{k+j=l}\frac{l!}{k!j!}h_{1}^{k}h_{2}^{j}k!j! \\
&\leq &Ch^{l}l!,
\end{eqnarray*}%
where $C=2^{l}C_{1}C_{2}>0$ and $h=h_{1}+h_{2}>0.$ It remains to show that $%
\mathcal{B}_{aa}^{\left\{ 1\right\} }$ is closed in $\mathcal{D}_{L^{\infty
}}^{\left\{ 1\right\} }$. Let $\left( \varphi _{n}\right) _{n\in
\mathbb{N}
}$ be a sequence of $\mathcal{B}_{aa}^{\left\{ 1\right\} }$ that converges
to $\varphi $ $\in $ $\mathcal{D}_{L^{\infty }}^{\left\{ 1\right\} },$ i.e. $%
\exists h>0$, $\exists C>0$, $\forall n\in
\mathbb{N}
,\forall j\in
\mathbb{Z}
_{+}$, $\left\Vert \varphi _{n}^{(j)}-\varphi ^{(j)}\right\Vert _{\infty
}\leq Ch^{j}j!.$ Due to the uniform convergence, and by Proposition $\ref%
{pro}$-$\left( \ref{pro1}\right) ,$ $\varphi ^{(j)}\in \mathcal{C}%
_{aa},\forall j\in
\mathbb{Z}
_{+}$. Consequently, we have $\varphi \in \mathcal{B}_{aa}^{\left\{
1\right\} }.$

$\mathit{2.}$ If \ $f\in \mathcal{B}_{aa}^{\left\{ 1\right\} }$ and $g\in
L^{1}$, then $f\ast g\in C^{\infty }$, and $\forall j\in
\mathbb{Z}
_{+}$, $\left( f\ast g\right) ^{(j)}=f^{(j)}\ast g\in \mathcal{C}_{aa}$ by
Proposition $\ref{pro}$-$\left( \ref{pro4}\right) $. On the other hand, $%
\exists h>0$, $\exists C>0,$ $\forall j\in
\mathbb{Z}
_{+}$, $\forall x\in
\mathbb{R}
,$%
\begin{equation*}
\left\vert \left( f\ast g\right) ^{(j)}(x)\right\vert \leq \left\Vert
f^{(j)}\right\Vert _{\infty }\int\limits_{%
\mathbb{R}
}\left\vert g(x-y)\right\vert dy\leq Ch^{j}j!\left\Vert g\right\Vert _{1},
\end{equation*}%
hence $f\ast g\in \mathcal{B}_{aa}^{\left\{ 1\right\} }.$

$\mathit{3}.$ If $F\in \mathcal{B}_{aa}^{\left\{ 1\right\} }$ is a primitive
of $f\in \mathcal{B}_{aa}^{\left\{ 1\right\} }$, so $F$ is bounded.
Conversely, let $F$ \ be a bounded primitive of $f\in \mathcal{B}%
_{aa}^{\left\{ 1\right\} },$ by Proposition $\ref{pro}$-$(\ref{pro6}),$ we
have $F\in \mathcal{C}_{aa}.$ Moreover, since $f\in \mathcal{B}%
_{aa}^{\left\{ 1\right\} },$ then $\exists h>0,$ $\exists C>0,$ $\forall $ $%
j\in
\mathbb{N}
,$
\begin{equation*}
\left\Vert F^{(j)}\right\Vert _{\infty }=\left\Vert f^{(j-1)}\right\Vert
_{\infty }\leq Ch^{j-1}(j-1)!\leq \frac{C}{h}h^{j}j!,
\end{equation*}%
so $F\in \mathcal{D}_{L^{\infty }}^{\left\{ 1\right\} }$ and due to the
assertion $2$ of Proposition \ref{baa1}$\mathit{,}$ we get $F\in \mathcal{B}%
_{aa}^{\left\{ 1\right\} }.$
\end{proof}

\section{Spaces of hyperfunctions}

In this section we recall spaces of hyperfunctions and prove some
of their properties.

Let $\mathcal{F}_{h,k}$ denotes the space of all infinitely derivable
functions $\varphi $ such that
\begin{equation*}
\left\Vert \varphi \right\Vert _{h,k}:=\sup_{\substack{ x\in
\mathbb{R}
\\ j\in
\mathbb{Z}
_{+}}}\frac{\left\vert e^{k\left\vert x\right\vert }\varphi
^{(j)}(x)\right\vert }{h^{j}j!}<\infty ,
\end{equation*}%
then $\left( \mathcal{F}_{h,k},\left\Vert .\right\Vert _{h,k}\right) ,h,k>0,$
is a Banach space. The space $\mathcal{F}:=\cup _{h,k>0}\mathcal{F}_{h,k}$
endowed with the topology of inductive limit of the Banach spaces $\left(
\mathcal{F}_{h,k}\right) _{h,k}$ is a $\left( DF\right) $ space, for more
details see \cite{Perio, Gelfand, Pili2}.

\begin{remark}
\label{RMK}

\begin{enumerate}
\item The space $\mathcal{F}$\ is in fact the space $S_{1}^{1}$ of \cite%
{Gelfand}.

\item \label{rmk}We have a continuous and dense embedding of $\mathcal{F}$
into $\mathcal{D}_{L^{1}}^{\left\{ 1\right\} }.$
\end{enumerate}
\end{remark}

Let $\mu >0,$ a differential operator of infinite order
\begin{equation*}
P(D_{x})=\sum\limits_{j=0}^{+\infty }a_{j}\left( \frac{d}{idx}\right) ^{j}
\end{equation*}%
is called a $\left\{ j!^{\mu }\right\} -$ultradifferential operator if for
every $L>0$ there exists $C>0$ such that
\begin{equation}
\left\vert a_{j}\right\vert \leq C\frac{L^{j}}{j!^{\mu }}\text{, \ \ }%
\forall j\in
\mathbb{Z}
_{+}  \label{ultra}
\end{equation}

\begin{proposition}
\label{P(D)ul}Let $P(D_{x})$ be a $\left\{ j!^{2}\right\} -$%
ultradifferential operator, then the operators $P(D_{x}^{2})$ and $P(D_{x})$
are continuous linear mappings from $\mathcal{F}$ into $\mathcal{F}$ and
from $\mathcal{D}_{L^{1}}^{\left\{ 1\right\} }$ into $\mathcal{D}%
_{L^{1}}^{\left\{ 1\right\} }.$
\end{proposition}

\begin{proof}
The linearity of $P(D_{x}^{2})$ is obvious. Let $\varphi \in $ $\mathcal{F}$%
, then there exist $h,k>0$ such that $\forall x\in
\mathbb{R}
,\left\vert e^{k\left\vert x\right\vert }\varphi ^{(j)}(x)\right\vert \leq
h^{j}j!\left\Vert \varphi \right\Vert _{h,k}$ for all $j\in
\mathbb{Z}
_{+}$. Using the inequalities $\left( \ref{ultra}\right) $, then
\begin{equation*}
\left\vert e^{k\left\vert x\right\vert }(P(D_{x}^{2})\varphi
)^{(j)}(x)\right\vert \leq \sum\limits_{\substack{ i\in \mathbb{Z}_{+}}}%
\dfrac{CL^{i}}{i!^{2}}h^{2i+j}(2i+j)!\left\Vert \varphi \right\Vert _{h,k},
\end{equation*}%
and since $(2i+j)!\leq 2^{4i+j}i!^{2}j!$, we obtain%
\begin{equation*}
\lvert e^{k\lvert x\rvert }(P(D_{x}^{2})\varphi )^{(j)}(x)\rvert \leq
C(2h)^{j}j!\left\Vert \varphi \right\Vert _{h,k}\sum\limits_{\substack{ i\in
\mathbb{Z}_{+}}}(2^{4}Lh^{2})^{i},
\end{equation*}%
if $L>0$\ such that $\ Lh^{2}<\dfrac{1}{2^{5}}$, it follows
\begin{equation*}
\left\Vert P(D^{2})\varphi \right\Vert _{2h,k}\leq C\left\Vert \varphi
\right\Vert _{h,k},
\end{equation*}%
which means the continuity of $P(D_{x}^{2})$ from $\mathcal{F}$ into $%
\mathcal{F}$. In the same way we obtain the continuity of $P(D_{x})$ from $%
\mathcal{F}$ into $\mathcal{F}$ and of $P(D_{x}^{2})$ and $P(D_{x})$ from $%
\mathcal{D}_{L^{1}}^{\left\{ 1\right\} }$ into $\mathcal{D}_{L^{1}}^{\left\{
1\right\} }.$
\end{proof}

\begin{remark}
\label{d}

\begin{enumerate}
\item \label{j!}If $P(D_{x})$ is a $\left\{ j!^{\mu }\right\} -$%
ultradifferential operator acting on $\mathcal{F}$ and $\mathcal{D}%
_{L^{1}}^{\left\{ 1\right\} }$, then for every $0<\lambda \leq \mu ,P(D_{x})$
is also a $\left\{ j!^{\lambda }\right\} -$ultradifferential operator acting
on $\mathcal{F}$ and $\mathcal{D}_{L^{1}}^{\left\{ 1\right\} }.$

\item \label{P}If $P(D_{x})$ is only a $\left\{ j!\right\} -$%
ultradifferential operator, then the operator $P(D_{x}^{2})$ does not
necessary operate on $\mathcal{F}$.

\item \label{P2}If $P(D_{x})$ is a $\left\{ j!\right\} -$ultradifferential
operator, then the operator $P(D_{x})$ is continuous linear mapping from $%
\mathcal{F}$ into $\mathcal{F}$ and from $\mathcal{D}_{L^{1}}^{\left\{
1\right\} }$ into $\mathcal{D}_{L^{1}}^{\left\{ 1\right\} }.$
\end{enumerate}
\end{remark}

Recall that the space of $L^{p}$-hyperfunctions, $1<p\leq +\infty ,$ denoted
by $\mathcal{D}_{L^{p},\left\{ 1\right\} }^{\prime },$ is defined as the
topological dual of $\mathcal{D}_{L^{q}}^{\left\{ 1\right\} }$, where $\frac{%
1}{p}+\frac{1}{q}=1,$ see \cite{Perio, Pili2}. The space $\mathcal{\dot{D}}%
_{L^{\infty }\text{ }}^{\left\{ 1\right\} }$denotes the closure in $\mathcal{%
D}_{L^{\infty }}^{\left\{ 1\right\} }$ of the space $\mathcal{F}$, and its
topological dual $\left( \mathcal{\dot{D}}_{L^{\infty }\text{ }}^{\left\{
1\right\} }\right) ^{\prime }$ is denoted by $\mathcal{D}_{L^{1},\left\{
1\right\} }^{\prime }.$

\begin{definition}
The topological dual $\mathcal{F}^{\prime }$ of $\mathcal{F}$ is called the
space of Fourier hyperfunctions. The space $\mathcal{D}_{L^{\infty },\left\{
1\right\} }^{\prime }:=\left( \mathcal{D}_{L^{1}}^{\left\{ 1\right\}
}\right) ^{\prime }$ is the space of bounded hyperfunctions.
\end{definition}

\begin{remark}
\label{D1}In view of Remark $\ref{RMK}$-$\left( \ref{rmk}\right) $, we have
a continuous embedding of $\ \mathcal{D}_{L^{\infty },\left\{ 1\right\}
}^{\prime }$ into $\mathcal{F}^{\prime }$.
\end{remark}

\begin{remark}
\label{DD}

\begin{enumerate}
\item \label{D3}Let $P(D_{x})$ be a $\left\{ j!^{2}\right\} -$%
ultradifferential operator, by Proposition $\ref{P(D)ul}$, it follows the
continuity of linear operators $P(D_{x}^{2})$ and $P(D_{x})$ from $\mathcal{D%
}_{L^{\infty },\left\{ 1\right\} }^{\prime }$into $\mathcal{D}_{L^{\infty
},\left\{ 1\right\} }^{\prime }$, and also from $\mathcal{F}^{\prime }$ into
$\mathcal{F}^{\prime }$.

\item \label{D4}If $P(D_{x})$ is a $\left\{ j!\right\} -$ultradifferential
operator, by Remark $\ref{d}$-$\left( \ref{P2}\right) ,$ it holds that the
continuity of linear operator $P(D_{x})$ from $\mathcal{D}_{L^{\infty
},\left\{ 1\right\} }^{\prime }$into $\mathcal{D}_{L^{\infty },\left\{
1\right\} }^{\prime }$, and also from $\mathcal{F}^{\prime }$ into $\mathcal{%
F}^{\prime }$.
\end{enumerate}
\end{remark}

\section{Almost automorphic hyperfunctions}
In this section we introduce almost automorphic hyperfunctions and give some
of their main properties.

The heat kernel $E$ is defined by%
\begin{equation*}
E(x,t)=\left\{
\begin{array}{l}
\frac{1}{\sqrt{4\pi t}}e^{-\frac{x^{2}}{4t}},t>0 \\
0,t\leq 0%
\end{array}%
\right.
\end{equation*}

We recall some needed results, see \cite{Perio, matzu2}.$\ $

\begin{lemma}
\label{cv}

\begin{enumerate}
\item The function$\ E(.,t)$ $\in \mathcal{F}$, $\forall $ $t>0$.

\item The function $\varphi \left( .\right) \ast E(.,t)$ $\in \mathcal{D}%
_{L^{1}}^{\left\{ 1\right\} }$, $\forall $ $\varphi \in \mathcal{D}%
_{L^{1}}^{\left\{ 1\right\} }$, $\forall $ $t>0.$

\item \label{cv3}We have $\varphi \left( .\right) \ast E(.,t)\longrightarrow
\varphi \left( .\right) $ in $\mathcal{D}_{L^{1}}^{\left\{ 1\right\} }$ as $%
t\rightarrow 0^{+}.$
\end{enumerate}
\end{lemma}

\begin{definition}
The Gauss transform of \ $U\in \mathcal{D}_{L^{\infty },\left\{ 1\right\}
}^{\prime }$ is defined and denoted by \
\begin{equation*}
u(.,t)=U(.)\ast E(.,t)=<U_{y},E(.-y,t)>,\text{\ }t>0
\end{equation*}
\end{definition}

\begin{remark}
The Gauss transform of \ $U\in \mathcal{D}_{L^{\infty },\left\{ 1\right\}
}^{\prime }$ is a $C^{\infty }$ function in $%
\mathbb{R}
\times
\mathbb{R}
_{+}.$
\end{remark}

For $h\in
\mathbb{R}
,\tau _{h}$ denotes the translation operator, i.e. $\tau _{h}\varphi
(.):=\varphi (.+h)$ for a function $\varphi .$ For a hyperfuntion $U\in
\mathcal{F}^{\prime },$ $\tau _{h}U$ is defined by $<\tau _{h}U,\varphi
>=<U,\tau _{-h}\varphi >,$ $\varphi \in \mathcal{F}$.

We recall the following important result needed in the proof of the next
theorem.

\begin{lemma}
\cite{matzu2}\label{lem} For any $h>0$ and $\epsilon >0$, there exist
functions $v\in C_{c}^{\infty }(\left[ 0,\epsilon \right] ),$ $w\in
C_{c}^{\infty }(\left[ \frac{\epsilon }{2},\epsilon \right] )$ and an
ultradifferential operator $P\left( \frac{d}{dt}\right)
=\sum\limits_{j=0}^{+\infty }a_{j}\left( \frac{d}{dt}\right) ^{j}$ such that%
\begin{eqnarray*}
\ \left\vert v^{(j)}(t)\right\vert &\leq &Ch^{-j}j!^{2},\forall j\in
\mathbb{Z}
_{+}, \\
\left\vert v(t)\right\vert &\leq &Ce^{-\frac{h}{t}},\text{ }0<t<\infty ,
\end{eqnarray*}%
\begin{equation*}
\forall h_{1}<h,\text{ }\exists C_{1}>0,\forall j\in
\mathbb{Z}
_{+},\left\vert a_{j}\right\vert \leq C_{1}\frac{h_{1}^{j}}{j!^{2}},\text{ \
}
\end{equation*}%
\begin{equation*}
P\left( \frac{d}{dt}\right) v(t)+w(t)=\delta
\end{equation*}
\end{lemma}

To introduce almost automorphic hyperfunctions, we prove the main results of
this section.

\begin{theorem}
\label{thm*}Let $U\in \mathcal{\ D}_{L^{\infty },\left\{ 1\right\} }^{\prime
}$, the following statements are equivalent :

\begin{enumerate}
\item \label{th1}$U\ast \varphi \in \mathcal{\ C}_{aa}$, $\forall \varphi
\in \mathcal{F}$.

\item \label{th2}There exist two functions $f$, $g$ belonging to $\mathcal{C}%
_{aa}$ and a $\left\{ j!^{2}\right\} -$ultradifferential operator $P\left(
D_{x}\right) $ such that $U=P\left( D_{x}^{2}\right) f+g$ in $\mathcal{D}%
_{L^{\infty },\left\{ 1\right\} }^{\prime }$.

\item \label{th3}Every sequence $\left( s_{m}\right) _{m\in
\mathbb{N}
}\subset
\mathbb{R}
$ admits a subsequence $\left( s_{m_{k}}\right) _{k}$ such that%
\begin{equation*}
V:=\lim_{k\rightarrow +\infty }\tau _{s_{m_{k}}}U\text{ \ exists in }%
\mathcal{D}_{L^{\infty },\left\{ 1\right\} }^{\prime }\text{, }
\end{equation*}%
and%
\begin{equation*}
\lim_{k\rightarrow +\infty }\tau _{-s_{m_{k}}}V=U\text{ in \ }\mathcal{D}%
_{L^{\infty },\left\{ 1\right\} }^{\prime }
\end{equation*}

\item Every sequence $\left( s_{m}\right) _{m\in
\mathbb{N}
}\subset
\mathbb{R}
$, admits a subsequence $\left( s_{m_{k}}\right) _{k}$ such that%
\begin{equation}
\lim_{l\rightarrow +\infty }\lim_{k\rightarrow +\infty }\tau
_{-s_{m_{l}}}\tau _{s_{m_{k}}}U\text{ }=U\text{\ \ \ in }\mathcal{D}%
_{L^{\infty },\left\{ 1\right\} }^{\prime }  \label{eq}
\end{equation}

\item The Gauss transform of $U$ is an almost automorphic function in $x\in
\mathbb{R}
$ for every $t>0$.

\item \label{th6}There exists a sequence $\left( \varphi _{n}\right)
_{n}\subset \mathcal{B}_{aa}^{\left\{ 1\right\} }$ converging to $U$ in $%
\mathcal{D}_{L^{\infty },\left\{ 1\right\} }^{\prime }$.
\end{enumerate}
\end{theorem}

\begin{proof}
$\mathit{1.\Rightarrow 2.}$ \ Let $v,w$ and $P\left( D_{t}\right) $ be as in
Lemma$\ \ref{lem}$, the functions \ $V$ and $W$ defined by
\begin{equation*}
V(x,t):=\int_{0}^{+\infty }v(s)E(x,t+s)ds\text{\ \ \ \ }W(x,t):=\int_{0}^{+%
\infty }w(s)E(x,t+s)ds,
\end{equation*}%
belong to $C^{\infty }$ in $%
\mathbb{R}
\times
\mathbb{R}
^{+},$ and for every fixed $t>0$ they belong to $\mathcal{F\subset }$ $%
\mathcal{D}_{L^{1}}^{\left\{ 1\right\} }$ as functions of $x$, furthermore
we have
\begin{equation*}
P\left( -\frac{d}{dt}\right) V(x,t)=E(x,t)-W(x,t),
\end{equation*}%
see \cite{Fap}. As $P\left( -\frac{d}{dt}\right) U\in \mathcal{D}%
_{L^{\infty },\left\{ 1\right\} }^{\prime }$ and $V\in \mathcal{D}%
_{L^{1}}^{\left\{ 1\right\} }$, then in view of Corollary 5.10 of \cite%
{Pili2}, we have
\begin{equation*}
\left( P\left( -\frac{d}{dt}\right) U(.)\right) \ast V(x,t)=P\left( -\frac{d%
}{dt}\right) \left( U(.)\ast V(x,t)\right) ,
\end{equation*}%
so
\begin{equation}
P\left( -\frac{d}{dt}\right) \left( U(.)\ast V(x,t)\right) +U(.)\ast
W(x,t)=U(.)\ast E(x,t)  \label{num1}
\end{equation}%
As the heat kernel $E$ is a classical solution of the heat equation in $%
\mathbb{R}
\times
\mathbb{R}
^{+},$ it follows that $V$ is also a classical solution of the heat equation
in $%
\mathbb{R}
\times
\mathbb{R}
^{+}$, so we have
\begin{eqnarray*}
P\left( -\frac{d}{dt}\right) \left( U(.)\ast V(x,t)\right)
&=&\sum\limits_{j=0}^{+\infty }a_{j}\left( U(.)\ast \left( -\frac{d}{dt}%
\right) ^{j}V(x,t)\right) \\
&=&\sum\limits_{j=0}^{+\infty }a_{j}\left( U(.)\ast \left( -\frac{d^{2}}{%
dx^{2}}\right) ^{j}V(x,t)\right) \\
&=&P(D_{x}^{2})\left( U(.)\ast V(x,t)\right) \text{,}
\end{eqnarray*}%
consequently,
\begin{equation}
P\left( D_{x}^{2}\right) \left( U(.)\ast V(x,t)\right) +U(.)\ast
W(x,t)=U(.)\ast E(x,t)  \label{D}
\end{equation}%
As $W$ is also a classical solution of the heat equation in $%
\mathbb{R}
\times
\mathbb{R}
^{+}$, then $U\ast V$ and $U\ast W$ are classical solutions of the heat
equation in $%
\mathbb{R}
\times
\mathbb{R}
^{+}$. Moreover, by hypothesis we have that for every fixed $t>0$ the
functions $U(.)\ast V(.,t)$, $U(.)\ast W(.,t)$ belong to $\mathcal{C}_{aa}$,
then they are continuous and bounded functions on $%
\mathbb{R}
$, it follows by Theorem $1$ Chapter $VII$ of \cite{widder}, there exist
functions $f$, $g$ $\in $ $L^{\infty }\left(
\mathbb{R}
\right) $ such that
\begin{equation*}
U(.)\ast V(x,t)=f\ast E(x,t)\text{ \ and }U(.)\ast W(x,t)=g\ast E(x,t)
\end{equation*}%
Due to $\left( \ref{D}\right) $, we obtain%
\begin{equation}
P\left( D_{x}^{2}\right) \left( f\ast E(x,t)\right) +g\ast E(x,t)=U(.)\ast
E(x,t)  \label{DD1}
\end{equation}%
For every $t>0$, as the functions $f\ast E(.,t)=U(.)\ast V(.,t)$ and $g\ast
E(.,t)=U(.)\ast W(.,t),$ then they belong to $\mathcal{C}_{aa}$ by
hypothesis. We have%
\begin{eqnarray*}
\left\vert f\ast E(x,t)-f(x)\right\vert &=&\left\vert \frac{1}{\sqrt{4\pi t}}%
\int_{%
\mathbb{R}
}e^{\frac{-y^{2}}{4t}}\left( f(x-y)-f(x)\right) dy\right\vert , \\
&\leq &\frac{1}{\sqrt{\pi }}\int_{%
\mathbb{R}
}e^{-r^{2}}\left\vert f(x-r\sqrt{4t})-f(x)\right\vert dr,
\end{eqnarray*}%
and using dominated convergence Theorem, we obtain
\begin{equation*}
\sup_{x\in
\mathbb{R}
}\left\vert f\ast E(x,t)-f(x)\right\vert \underset{t\rightarrow 0^{+}}{%
\rightarrow 0},
\end{equation*}%
i.e. $f\ast E(.,t)$ converge uniformly to $f$ on $%
\mathbb{R}
$ as $t\rightarrow 0^{+}.$ In the same way, we obtain that $g\ast E(.,t)$
converge uniformly to $g$ on $%
\mathbb{R}
$ as $t\rightarrow 0^{+}.$ Due Proposition $\ref{pro}$-$\left( \ref{pro1}%
\right) $ gives $f$, $g\in \mathcal{C}_{aa}$. Consequently, and as the
linear operator $P\left( D_{x}^{2}\right) $ is continuous from $\mathcal{D}%
_{L^{\infty },\left\{ 1\right\} }^{\prime }$into $\mathcal{D}_{L^{\infty
},\left\{ 1\right\} }^{\prime }$, it follows that $P\left( D_{x}^{2}\right)
\left( f\ast E(.,t)\right) $ converge to $P\left( D_{x}^{2}\right) f$ \ in $%
\mathcal{D}_{L^{\infty },\left\{ 1\right\} }^{\prime }$ as $t\rightarrow
0^{+}.$ Due to Lemma $\ref{cv}$-$\left( \ref{cv3}\right) $, we have that for
every $\varphi \in \mathcal{D}_{L^{1}}^{\left\{ 1\right\} },$
\begin{equation*}
\lim_{t\rightarrow 0^{+}}<U(.)\ast E(.,t),\varphi (.)>=\lim_{t\rightarrow
0^{+}}<U_{y},E(.-y,t)\ast \varphi (.)>=<U,\varphi (.)>,
\end{equation*}%
which gives that $\underset{t\rightarrow 0^{+}}{\lim }U(.)\ast E(.,t)=U$ in $%
\mathcal{D}_{L^{\infty },\left\{ 1\right\} }^{\prime }$. Finally, by $\left( %
\ref{DD1}\right) $ we obtain $\mathit{2}$.

$\mathit{2.\Rightarrow 3.}$ Suppose that there exist $f$, $g\in $ $\mathcal{C%
}_{aa}$ and $P\left( D_{x}\right) $ a $\left\{ j!^{2}\right\} -$%
ultradifferential operator such that $U=P\left( D_{x}^{2}\right) f+g$. For
every sequence $\left( s_{m}\right) _{m\in
\mathbb{N}
}\subset
\mathbb{R}
$, there exists a subsequence $\left( s_{m_{k}}\right) _{k}$ such that $%
\forall x\in
\mathbb{R}
,$%
\begin{equation*}
\lim_{k\rightarrow +\infty }f(x+s_{m_{k}})=F(x)\text{ \ \ and \ \ \ \ \ \ }%
\lim_{k\rightarrow +\infty }F(x-s_{m_{k}})=f(x)\text{,}
\end{equation*}%
\begin{equation*}
\lim_{k\rightarrow +\infty }g(x+s_{m_{k}})=G(x)\text{ \ \ and \ \ \ \ \ \ }%
\lim_{k\rightarrow +\infty }G(x-s_{m_{k}})=g(x)\text{,}
\end{equation*}%
where the functions $F$, $G\in L^{\infty }.$ Then for every $\varphi \in
\mathcal{D}_{L^{1}}^{\left\{ 1\right\} }$, we have%
\begin{equation*}
\lim_{k\rightarrow +\infty }\langle \tau _{s_{m_{k}}}U,\varphi \rangle
=\lim_{k\rightarrow +\infty }\int_{%
\mathbb{R}
}f(x+s_{m_{k}})P(D_{x}^{2})\varphi (x)dx+\lim_{k\rightarrow +\infty }\int_{%
\mathbb{R}
}g(x+s_{m_{k}})\varphi (x)dx,
\end{equation*}%
and using dominated convergence Theorem, it follows that%
\begin{equation*}
\lim_{k\rightarrow +\infty }\langle \tau _{s_{m_{k}}}U,\varphi \rangle
=<V,\varphi >,
\end{equation*}%
where $V$ $:=$ $P(D_{x}^{2})F+G\in \mathcal{D}_{L^{\infty },\left\{
1\right\} }^{\prime }$. In the same way, for every $\varphi \in \mathcal{D}%
_{L^{1}}^{\left\{ 1\right\} }$, we have%
\begin{eqnarray*}
\lim_{k\rightarrow +\infty }\langle \tau _{-s_{m_{k}}}V,\varphi \rangle
&=&\int_{%
\mathbb{R}
}\lim_{k\rightarrow +\infty }F(x-s_{m_{k}})P(D_{x}^{2})\varphi (x)dx+\int_{%
\mathbb{R}
}\lim_{k\rightarrow +\infty }G(x-s_{m_{k}})\varphi (x)dx \\
&=&\langle U,\varphi \rangle \text{,}
\end{eqnarray*}%
i.e. $\lim_{k\rightarrow +\infty }\tau _{-s_{m_{k}}}V=U$ in $\mathcal{D}%
_{L^{\infty },\left\{ 1\right\} }^{\prime }$.

$\mathit{3.\Longrightarrow 4.}$ Let $3$ holds, for every $\varphi \in
\mathcal{D}_{L^{1}}^{\left\{ 1\right\} }$ we have
\begin{equation*}
\lim_{l\rightarrow +\infty }\lim_{k\rightarrow +\infty }\langle \tau
_{-s_{m_{l}}}\tau _{s_{m_{k}}}U,\varphi \rangle =\lim_{l\rightarrow +\infty
}\langle V,\tau _{s_{m_{l}}}\varphi \rangle =\lim_{l\rightarrow +\infty
}\langle \tau _{-s_{m_{l}}}V,\varphi \rangle =\langle U,\varphi \rangle ,
\end{equation*}%
i.e. $\lim_{l\rightarrow +\infty }\lim_{k\rightarrow +\infty }\tau
_{-s_{m_{l}}}\tau _{s_{m_{k}}}U$ $=U$\ in $\mathcal{D}_{L^{\infty },\left\{
1\right\} }^{\prime }$.

$\mathit{4.\rightarrow 5.}$ Let a sequence $\left( s_{m}\right) _{m\in
\mathbb{N}
}\subset
\mathbb{R}
$, with a subsequence $\left( s_{m_{k}}\right) _{k}$ such that ($\ref{eq}$)
holds. For fixed $t>0,$%
\begin{eqnarray*}
\lim_{l\rightarrow +\infty }\lim_{k\rightarrow +\infty }\tau
_{-s_{m_{l}}}\tau _{s_{m_{k}}}u(.,t) &=&\lim_{l\rightarrow +\infty
}\lim_{k\rightarrow +\infty }\langle \tau _{-s_{m_{l}}}\tau
_{s_{m_{k}}}U_{y},E(.-y,t)\rangle \\
&=&\langle U_{y},E(.-y,t)\rangle =u(.,t),
\end{eqnarray*}%
hence by Remark $\ref{limlim}$, the Gauss transform $u(.,t)$ is a classical
almost automorphic as function of $x$.

$\mathit{5.\rightarrow 6.}$ As $U\in \mathcal{\ D}_{L^{\infty },\left\{
1\right\} }^{\prime }$, i.e. $\forall h>0,\exists A>0,\forall j\in
\mathbb{Z}
_{+},\forall x\in
\mathbb{R}
,$
\begin{eqnarray*}
\left\vert \left( U\left( .\right) \ast E(.,t)\right) ^{(j)}(x)\right\vert
&\leq &A\left\Vert E^{(j)}(x-.,t)\right\Vert _{L^{1},h}, \\
&\leq &A\sup_{i\in
\mathbb{Z}
_{+}}\frac{\left\Vert E^{(i+j)}(x-.,t)\right\Vert _{1}}{h^{i}i!}
\end{eqnarray*}%
since $(i+j)!\leq 2^{i+j}i!j!$, we have $\forall k\in
\mathbb{Z}
_{+},$
\begin{equation*}
\left\vert \left( U\left( .\right) \ast E(.,t)\right) ^{(j)}(x)\right\vert
\leq Ah^{j}j!\sup_{k\in
\mathbb{Z}
_{+}}\frac{2^{k}\left\Vert E^{(k)}(x-.,t)\right\Vert _{1}}{h^{k}k!}
\end{equation*}%
from Proposition $1.1$ of \cite{matzu2}, we have there exist $C>0,a\in \left]
0,1\right[ $, where $a$ can be taken as close as desired to $1$, such that
for every $t>0,$
\begin{equation*}
\left\Vert E^{(k)}(x-.,t)\right\Vert _{1}\leq 2\left( \frac{\pi }{a}\right)
^{\frac{1}{2}}\left( Ct^{-\frac{1}{2}}\right) ^{k}k!^{\frac{1}{2}}
\end{equation*}%
Hence, $\forall h>0,\exists A>0,\exists C>0,\exists a\in \left] 0,1\right[
,\forall j\in
\mathbb{Z}
_{+},\forall t>0,$
\begin{eqnarray*}
\left\vert \left( U\left( .\right) \ast E(.,t)\right) ^{(j)}(x)\right\vert
&\leq &2A\left( \frac{\pi }{a}\right) ^{\frac{1}{2}}h^{j}j!\sup_{k\in
\mathbb{Z}
_{+}}\frac{\left( 2Ct^{-\frac{1}{2}}\right) ^{k}}{h^{k}k!^{\frac{1}{2}}} \\
&\leq &2A\left( \frac{\pi }{a}\right) ^{\frac{1}{2}}h^{j}j!\left( \sup_{k\in
\mathbb{Z}
_{+}}\left( \frac{2^{2}C^{2}}{th^{2}}\right) ^{k}\frac{1}{k!}\right) ^{\frac{%
1}{2}} \\
&\leq &2A\left( \frac{\pi }{a}\right) ^{\frac{1}{2}}h^{j}j!\exp \left( \frac{%
2C^{2}}{th^{2}}\right)
\end{eqnarray*}
Consequently, it holds
\begin{equation*}
\left\Vert \left( U\left( .\right) \ast E(.,t)\right) ^{(j)}\right\Vert
_{\infty }\leq B(t)h^{j}j!,
\end{equation*}%
where $B(t)=2A\left( \frac{\pi }{a}\right) ^{\frac{1}{2}}\exp \left( \frac{%
2C^{2}}{th^{2}}\right) ,$ which gives $U(.)\ast E(.,t)\in \mathcal{D}%
_{L^{\infty }}^{\left\{ 1\right\} }$ for every fixed $t>0,$ and as by
assumption we have $U(.)\ast E(.,t)\in \mathcal{C}_{aa}$ for every fixed $%
t>0,$ it follows that $U(.)\ast E(.,t)\in \mathcal{B}_{aa}^{\left\{
1\right\} }$ by Proposition $\ref{baa1}$-$\left( \ref{baa12}\right) .$

For every $n\in
\mathbb{N}
,$ the function $\varphi _{n}=U(.)\ast E(.,\frac{1}{n})\in \mathcal{B}%
_{aa}^{\left\{ 1\right\} }$, due to Lemma $\ref{cv}$-$\left( \ref{cv3}%
\right) $, we have that for every $\psi \in \mathcal{D}_{L^{1}}^{\left\{
1\right\} },$
\begin{equation*}
\lim_{n\rightarrow +\infty }<U(.)\ast E(.,\frac{1}{n}),\psi
(.)>=\lim_{n\rightarrow +\infty }<U_{y},E(.-y,\frac{1}{n})\ast \psi
(.)>=<U,\psi (.)>,
\end{equation*}%
which gives that $\underset{n\rightarrow +\infty }{\lim }\varphi _{n}=$ $U$
in $\mathcal{D}_{L^{\infty },\left\{ 1\right\} }^{\prime }.$

$\mathit{6.\rightarrow 1.}$ For any bounded subset $B\subset \mathcal{D}%
_{L^{1}}^{\left\{ 1\right\} }$ by hypothesis, we have
\begin{equation*}
\sup_{\psi \in B}\left\vert <\varphi _{n}-U,\psi >\right\vert \underset{%
n\rightarrow +\infty }{\rightarrow }0.
\end{equation*}%
It is clear that for every $\varphi \in \mathcal{F\subset }L^{1},$ we have $%
\varphi _{n}\ast \varphi \in \mathcal{B}_{aa}^{\left\{ 1\right\} }$ by
Proposition $\ref{Baa1}$-$\left( \ref{Baa12}\right) ,$ and the set $%
B:=\left\{ \tau _{-x}\check{\varphi}:x\in
\mathbb{R}
\right\} $ is a bounded in $\mathcal{D}_{L^{1}}^{\left\{ 1\right\} },$ where
$\check{\varphi}(y)=\varphi (-y),$ then
\begin{eqnarray*}
\sup_{x\in
\mathbb{R}
}\left\vert \left( \varphi _{n}\ast \varphi \right) (x)-\left( U\ast \varphi
\right) (x)\right\vert &=&\sup_{x\in
\mathbb{R}
}\left\vert <\varphi _{n}-U,\tau _{-x}\check{\varphi}>\right\vert , \\
&=&\sup_{\psi \in B}\left\vert <\varphi _{n}-U,\psi >\right\vert \underset{%
n\rightarrow +\infty }{\longrightarrow }0,
\end{eqnarray*}%
i.e. the sequence $\left( \varphi _{n}\ast \varphi \right) _{n}\subset
\mathcal{B}_{aa}^{\left\{ 1\right\} }$ converge uniformly to $U\ast \varphi $
as $n\rightarrow +\infty $ on $%
\mathbb{R}
,$ by Proposition $\ref{pro}$-$\left( \ref{pro1}\right) $, $U\ast \varphi
\in \mathcal{C}_{aa}$, $\forall \varphi \in \mathcal{F}.$
\end{proof}

\begin{definition}
A bounded hyperfunction $U$ is said almost automorphic if it satisfies any $%
\left( \text{hence every}\right) $ assertion of the above Theorem. We denote
by $\mathcal{F}_{aa}^{\prime }$ the space of all almost automorphic
hyperfunctions on $%
\mathbb{R}
$.
\end{definition}

\begin{example}
We have $\mathcal{F}_{ap}^{\prime }$ $\subsetneq \mathcal{F}_{aa}^{\prime }$%
, where $\mathcal{F}_{ap}^{\prime }$ denotes the space of almost periodic
hyperfunctions introduced in \cite{Fap}. Indeed, it is well known that $%
\mathcal{C}_{ap}\varsubsetneq \mathcal{C}_{aa}$, it follows $\mathcal{F}%
_{ap}^{\prime }$ $\subset \mathcal{F}_{aa}^{\prime }$. Let $f$ be an almost
automorphic function which is not almost periodic, and let $P(D_{x})$ be a $%
\left\{ j!^{2}\right\} -$ultradifferential operator, then $P\left(
D_{x}^{2}\right) f\in \mathcal{F}_{aa}^{\prime }\backslash \mathcal{F}%
_{ap}^{\prime }.$
\end{example}

We summarize the main properties of $\mathcal{F}_{aa}^{\prime }$.

\begin{proposition}
\label{proFaa}

\begin{enumerate}
\item \label{aaaa}The space $\mathcal{F}_{aa}^{\prime }$ is stable under $%
\left\{ j!\right\} -$ultradifferential operators and translations.

\item $\mathcal{B}_{aa}^{\left\{ 1\right\} }\times \mathcal{F}_{aa}^{\prime
}\subset \mathcal{F}_{aa}^{\prime }.$

\item \label{aa1}$\mathcal{F}_{aa}^{\prime }$ $\ast \mathcal{D}%
_{L^{1},\left\{ 1\right\} }^{\prime }\subset \mathcal{F}_{aa}^{\prime }.$

\end{enumerate}
\end{proposition}

\begin{proof}
$\mathit{1.}$ Let $U\in \mathcal{F}_{aa}^{\prime }$ and $P\left(
D_{x}\right) $ be a $\left\{ j!\right\} -$ultradifferential operator, due to
Remark $\ref{DD}$-$\left( \ref{D4}\right) $, we have $P\left( D_{x}\right)
U\in \mathcal{D}_{L^{\infty },\left\{ 1\right\} }^{\prime }$. By Corollary $%
5.10$ of \cite{Pili2}, we have
\begin{equation*}
\text{ }P\left( D_{x}\right) U\left( .\right) \ast \varphi =U\left( .\right)
\ast P\left( D_{x}\right) \varphi ,\varphi \in \mathcal{F}\subset \mathcal{D}%
_{L^{1}}^{\left\{ 1\right\} },
\end{equation*}%
hence by Remark $\ref{d}-\left( \ref{P2}\right) $and Theorem $\ref{thm*}$-$%
\left( \ref{th1}\right) $, it holds that $P\left( D_{x}\right) U\ast \varphi
\in \mathcal{C}_{aa}$, $\forall \varphi \in \mathcal{F},$ so $P\left(
D_{x}\right) U\in \mathcal{F}_{aa}^{\prime }.$

Let $U\in \mathcal{F}_{aa}^{\prime },$ then for every $\varphi \in \mathcal{F%
}$, $\tau _{h}U\ast \varphi =\tau _{h}\left( U\ast \varphi \right) $, for
every $h\in
\mathbb{R}
$. As $U\ast \varphi \in \mathcal{C}_{aa}$ and the space $\mathcal{C}_{aa}$
is invariant by translations, then $\tau _{h}\left( U\ast \varphi \right)
\in \mathcal{C}_{aa}$, so $\tau _{h}U\in \mathcal{F}_{aa}^{\prime }$, $%
\forall h\in
\mathbb{R}
$.

$\mathit{2.}$ Let $\psi \in \mathcal{B}_{aa}^{\left\{ 1\right\} }$ and $U\in
$ $\mathcal{F}_{aa}^{\prime },$ then there exists a sequence $\left( \varphi
_{n}\right) _{n}\subset \mathcal{B}_{aa}^{\left\{ 1\right\} }$ converging to
$U$ in $\mathcal{D}_{L^{\infty },\left\{ 1\right\} }^{\prime }$. Define $%
\psi _{n}:=\psi \varphi _{n}\subset \mathcal{B}_{aa}^{\left\{ 1\right\} },$
we have for every $\varphi \in \mathcal{D}_{L^{1}}^{\left\{ 1\right\} },$%
\begin{equation*}
\lim_{n\rightarrow +\infty }<\psi \varphi _{n},\varphi >=<\psi U,\varphi >,
\end{equation*}%
which gives that $\underset{n\rightarrow +\infty }{\lim }\psi _{n}=$ $\psi U$
in $\mathcal{D}_{L^{\infty },\left\{ 1\right\} }^{\prime }.$ Consequently,
by Theorem $\ref{thm*}$-$\left( \ref{th6}\right) $, it holds that $\psi U\in
\mathcal{F}_{aa}^{\prime }$ $.$

$\mathit{3.}$ Let $U\in \mathcal{F}_{aa}^{\prime }$, then there exist $f$, $%
g $ $\in $ $\mathcal{C}_{aa}$ and a $\left\{ j!^{2}\right\} -$%
ultradifferential operator $P\left( D_{x}\right) $ such that $U=P\left(
D_{x}^{2}\right) f+g$ in $\mathcal{D}_{L^{\infty },\left\{ 1\right\}
}^{\prime }$. From Theorem $3.4$ of \cite{Pili2}$,$ we have if $V\in
\mathcal{D}_{L^{1},\left\{ 1\right\} }^{\prime }$, then there exist $v\in $ $%
L^{1\text{ }}$and a $\left\{ j!\right\} -$ultradifferential operator $%
Q\left( D_{x}\right) $ such that $V=Q\left( D_{x}\right) v$ in $\mathcal{D}%
_{L^{1},\left\{ 1\right\} }^{\prime }.$ Consequently, and in view of
Corollary $5.10$ of \cite{Pili2}, we have%
\begin{eqnarray*}
U\ast V &=&P(D_{x}^{2})f\ast Q(D_{x})v+g\ast Q(D_{x})v, \\
&=&P(D_{x}^{2})Q(D_{x})(f\ast v)+Q(D_{x})(g\ast v), \\
&=&Q(D_{x})\left[ P(D_{x}^{2})(f\ast v)+(g\ast v)\right] ,
\end{eqnarray*}%
so, as $f\ast v$ and $g\ast v$ belong to $\mathcal{C}_{aa}$ by Proposition $%
\ref{pro}$-$\left( \ref{pro4}\right) $, due to Theorem $\ref{thm*}$-$\left( %
\ref{th2}\right) ,$ it holds that $P\left( D_{x}^{2}\right) \left( f\ast
v\right) +g\ast v\in $ $\mathcal{F}_{aa}^{\prime }.$ Then from the assertion
$1,$ we have $U\ast V\in \mathcal{F}_{aa}^{\prime }.$

\end{proof}

\begin{proposition}
A primitive of an almost automorphic hyperfunction is
almost automorphic if and only if it is bounded.
\end{proposition}

\begin{proof}
If $V\in \mathcal{F}_{aa}^{\prime }$ is a primitive of $U\in
\mathcal{F}_{aa}^{\prime }$, then $V\in \mathcal{D}_{L^{\infty },\left\{
1\right\} }^{\prime }.$ Conversely, let $V\in \mathcal{D}_{L^{\infty
},\left\{ 1\right\} }^{\prime }$ be a primitive of $U\in \mathcal{F}%
_{aa}^{\prime },$ then $V\ast \varphi \in L^{\infty }$, for every $\varphi
\in $ $\mathcal{F}$, and we have
\begin{equation*}
\left( V\ast \varphi \right) ^{\prime }=V^{\prime }\ast \varphi =U\ast
\varphi \in \mathcal{C}_{aa},\forall \varphi \in \mathcal{F},
\end{equation*}%
i.e. $V\ast \varphi $ is a bounded primitive of the almost automorphic
function $U\ast \varphi .$ Thus by Proposition $\ref{pro}$-$\left( \ref{pro6}%
\right) $, $V\ast \varphi \in \mathcal{C}_{aa}$, $\forall \varphi \in
\mathcal{F}$, so $V\in \mathcal{F}_{aa}^{\prime }.$
\end{proof}

\begin{remark}
The result of Proposition is an extension to almost automorphic hyperfunctions of the classical result of
Bohl-Bohr on primitives for almost periodic functions and almost automorphic distributions.
\end{remark}

\end{document}